\documentclass{amsart}

\usepackage{enumerate}
\usepackage{amsmath}%
\usepackage{amsfonts}%
\usepackage{amssymb}%
\usepackage{graphicx}
\usepackage{mathrsfs}
\usepackage{hyperref}
\usepackage[margin=1in]{geometry}
\newtheorem{theorem}{Theorem}
\theoremstyle{plain}

\newtheorem{conjecture}[theorem]{Conjecture}
\newtheorem{corollary}[theorem]{Corollary}

\newtheorem{proposition}[theorem]{Proposition}

\numberwithin{equation}{section}
\numberwithin{theorem}{section}
\numberwithin{case}{section}

\numberwithin{subcase}{case}
\def\A{\mathcal{A}}
\def\B{\mathcal{B}}

\def\F{\mathcal{F}}

\def\h{\mathcal{H}}

\def \a{\alpha}
\def \e{\epsilon}

\begin{document}
\title[Perfect Matchings in Hypergraphs]{Perfect Matchings in Hypergraphs and the Erd\H{o}s matching conjecture}
\author{Jie Han}
\address{School of Mathematics, University of Birmingham, Birmingham, B15 2TT, UK}
\email{J.Han@bham.ac.uk, jhan@ime.usp.br}
%\address{Instituto de Matem\'{a}tica e Estat\'{\i}stica, Universidade de S\~{a}o Paulo, Rua do Mat\~{a}o 1010, 05508-090, S\~{a}o Paulo, Brazil}
%\email{jhan@ime.usp.br}
\thanks{The author is supported by FAPESP (2014/18641-5, 2015/07869-8).}
\date{\today}

\begin{abstract}
We prove a new upper bound for the minimum $d$-degree threshold for perfect matchings in $k$-uniform hypergraphs when $d<k/2$.
As a consequence, this determines exact values of the threshold when $0.42k \le d < k/2$ or when $(k,d)=(12,5)$ or $(17,7)$.
Our approach is to give an upper bound on the Erd\H{o}s Matching Conjecture and convert the result to the minimum $d$-degree setting by an approach of K\"uhn, Osthus and Townsend.
To obtain exact thresholds, we also apply a result of Treglown and Zhao.
\end{abstract}

\maketitle

\section{Introduction}

\subsection{Perfect matchings via minimum degree conditions}

Given $k\ge 2$, a \emph{$k$-uniform hypergraph} (or a \emph{$k$-graph}) $H$ is a pair $H=(V, E)$, where $V$ is a finite vertex set and $E$ is a family of $k$-element subsets of $V$.
Given a $k$-graph $H$ and a set $S$ of $d$ vertices in $V(H)$, $0\le d\le k-1$, we denote by $\deg_H(S)$ the number of edges of $H$ containing $S$.
The \emph{minimum $d$-degree} of $H$ then is defined as
\[
\delta_d(H) = \min\left\{ \deg_H(S): S\in \binom{V(H)}{d} \right\}.
\]
Note that $\delta_0(H)=|E(H)|$ is the number of edges of $H$.

A \emph{matching} $M$ in $H$ is a collection of disjoint edges of $H$.
The \emph{size} of $M$ is the number of edges in $M$.
We say $M$ is a \emph{perfect matching} if it has size $|V|/k$.
For integers $n,k,d,s$ satisfying $0\le d\le k-1$ and $0\le s \le n/k$, let $m_d^s(k,n)$ be the smallest integer $m$ such that every $n$-vertex $k$-graph $H$ with $\delta_d(H)\ge m$ has a matching of size $s$.
For simplicity, we write $m_d(k,n)$ for $m_d^{n/k}(k,n)$.
Throughout this note, $o(1)$ stands for some function that tends to $0$ as $n$ tends to infinity.
The following conjecture \cite{HPS, KuOs-survey} has received much attention in the last decade (see \cite{AFHRRS, CzKa, HPS, Han14_mat, Han15_mat, Khan1, Khan2, KO06mat, KOT, MaRu, Pik, RRS06mat, RRS09, TrZh12, TrZh13, TrZh15} and the recent surveys \cite{RR, zsurvey}). 

\begin{conjecture}\label{conj:mat}
For $1\le d\le k-1$ and $k \mid n$,
\[
m_d(k,n) = \left(\max \left\{ \frac12, 1- \left(1- \frac{1}{k} \right)^{k-d}  \right\} +o(1)\right)\binom{n-d}{k-d}.
\]
\end{conjecture}

We remark that the quantities in the lower bound of the conjecture come from two different constructions. 
The second term can be seen by the following $k$-graph.
%Let $V$ be a set of size $n$ and fix $S\subseteq V$ with $|S|=s < n/k$.
Let $G(s)$ be the $k$-graph on $V$ whose edges are all $k$-sets that intersect a fixed $S\subseteq V$ with $|S|=s < n/k$. Clearly $G(n/k-1)$ has no perfect matching.
% and $\delta_d(H(n/k-1)) = \binom{n-d}{k-d} - \binom{n-d-s}{k-d} = (1- (1- s/n)^{k-d}  +o(1))\binom{n-d}{k-d}$.

On the other hand, the quantity $1/2$ comes from the following parity construction.
%Let $V$ be a set of size $n$ with a partition $V_1\cup V_2$ such that $||V_1| - |V_2||\le 2$ and $|V_1|$ is odd. Let $H'$ be the $k$-graph on $V$ whose edges are all $k$-sets $e$ that contain an even number of vertices in $V_1$.
%Observe that $H'$ does not have a perfect matching and $\delta_d(H') = (\frac12 +o(1)) \binom{n-d}{k-d}$.
Given a partition $V$ into non-empty sets $A, B$, let $\B_{n,k}(A,B)$ (or $\overline{\B}_{n,k}(A,B)$) be the $k$-uniform hypergraph with vertex set $V$ and whose edge set consists of all $k$-element subsets of $V$ that contains an odd (or even) number of vertices in $A$.
%Let $\overline{\B}_{n,k}(A,B)$ be the complement of $\B_{n,k}(A,B)$ (so each of its edges contains an even number of vertices in $A$).
Define $\h_{\text{ext}}(n,k)$ to be the collection of all hypergraphs $\overline{\B}_{n,k}(A,B)$ where $|A|$ is odd, and all hypergraphs ${\B}_{n,k}(A,B)$ where $|A|-n/k$ is odd.
%Define $\h_{\text{ext}}(n,k)$ to be the collection of the following hypergraphs.
%First, $\h_{\text{ext}}(n,k)$ contains all hypergraphs $\overline{\B}_{n,k}(A,B)$ where $|A|$ is odd.
%Second, $\h_{\text{ext}}(n,k)$ contains all hypergraphs ${\B}_{n,k}(A,B)$ where $|A|-n/k$ is odd.
It is easy to see that no hypergraph in $\h_{\text{ext}}(n,k)$ contains a perfect matching (see \cite{TrZh12}).
Define $\delta(n,k,d)$ to be the maximum of the minimum $d$-degrees among all the hypergraphs in $\h_{\text{ext}}(n,k)$.
Note that $\delta(n,k,d)= (1/2+o(1)) \binom{n-d}{k-d}$ but the general formula is unknown (see \cite{TrZh12} for more discussion).

%Unfortunately, we do not know which exact values of $|V_1|$ and $|V_2|$ maximize the value of $\delta_d(H^j)$ -- this forms a challenging question, see \cite{TrZh12} for a discussion.

Given $k\ge 3$, R\"odl, Ruci\'nski and Szemer\'edi \cite{RRS09} showed that $m_{k-1}(k,n)=\delta(n,k,k-1)+1$ for large $n$.
Treglown and Zhao \cite{TrZh12, TrZh13} generalized their result and showed that $m_d(k,n)=\delta(n,k,d)+1$ for all $d\ge k/2$.
For $d<k/2$, Conjecture~\ref{conj:mat} has been verified \cite{AFHRRS, HPS, Khan1, Khan2, KOT, TrZh15} for only a few cases, i.e., for $(k,d)\in \{(3,1), (4,1),(5,1),(5,2), (6,2), (7,3)\}$.
Moreover, exact values of $m_d(k,n)$ are known for $(k,d)\in \{(3,1), (4,1),(5,2), (7,3)\}$.
In general for $d<k/2$, the following best known upper bound is due to K\"uhn, Osthus and Townsend \cite[Theorem 1.2]{KOTo}, which improves earlier results by H\`an, Person and Schacht \cite{HPS}, and Markstr\"om and Ruci\'nski \cite{MaRu}.

\begin{theorem}\cite{KOTo}\label{thm:KOTo}
Let $n$, $1\le d<k/2$ be such that $n, k, d, n/k\in \mathbb{N}$. Then
\[
m_d(k,n) \le \left( \frac{k-d}k - \frac{k-d-1}{k^{k-d}} + o(1) \right) \binom{n-d}{k-d}.
\]
\end{theorem}

In this paper we show the following new upper bound on $m_d(k,n)$ for $1\le d<k/2$.

%\begin{theorem}\label{thm:int_mat}
%Let $n, k\ge 3$, $1\le d< k/2$ be integers. Then
%\begin{equation*}
%m_d(k,n) \le \left( \max\left\{\frac12, 1+ \left( \frac{(k-d)(k-2d-1)}{(k-1)^2}-1 \right) \left( 1-\frac{1}{k} \right)^{k-d}\right\} + o(1) \right) \binom{n-d}{k-d}.
%\end{equation*}
%\end{theorem}

\begin{theorem}\label{thm:int_mat}
Let $n, k\ge 3$, $1\le d< k/2$ be integers and $n\in k\mathbb{N}$. Then
\begin{equation*}
m_d(k,n) \le \max\left\{\delta(n,k,d)+1, \left(g(k,d) + o(1) \right) \binom{n-d}{k-d}\right\},
\end{equation*}
where
\[
g(k,d) := 1- \left(1- \frac{(k-d)(k-2d-1)}{(k-1)^2} \right) \left( 1-\frac{1}{k} \right)^{k-d}.
\]
\end{theorem}

Here we compare the bounds in Theorems~\ref{thm:KOTo} and \ref{thm:int_mat}.
First consider the case $d=xk$ for some fixed $x\in (0,1/2)$.
Let $g(x) := \lim_{k\rightarrow \infty} g(k,x k)$ and clearly $g(x) = 1- \left(3x-2x^2 \right) e^{x-1}$.
Straightforward application of Calculus shows that $g(x) \le 1 - \frac{3}e x \approx 1 - 1.1x$.
Note that when $d=x k$ and $k$ tends to infinity, the corresponding coefficient in the bound of Theorem~\ref{thm:KOTo} becomes $1-x$.
So in this range, when $k$ is sufficiently large, our bound is better than that of Theorem~\ref{thm:KOTo}.
Second, by simply plugging in values of $k, d$, one can see that the bound in Theorem~\ref{thm:KOTo} is better for small values of $k$ or when $d$ is much smaller than $k$.

Theorem~\ref{thm:int_mat} also provides some new exact values of $m_d(k,n)$.

\begin{corollary}\label{cor:mdk}
Given $1\le d<k/2$, let $n\in k\mathbb{N}$ be sufficiently large. Then $m_d(k,n) =\delta(n,k,d)+1$ if $0.42k \le d < k/2$ or $(k,d)\in \{(12,5), (17,7)\}$.
\end{corollary}

%We will compare the bounds in Theorems~\ref{thm:KOTo} and \ref{thm:int_mat} and show Corollary~\ref{cor:mdk} in Section 2.
%We postpone the proof of Corollary~\ref{cor:mdk} to Section 2.
\begin{proof}%[Proof of Corollary~\ref{cor:mdk}]
%Recall that the result for $d\ge k/2$ has been obtained in \cite{TrZh13}.
For all cases, since $n$ is sufficiently large, by Theorem~\ref{thm:int_mat}, it suffices to show $g(k, d)<1/2$.
The cases when $k\le 20$ can be verified by hand.
For $k\ge 20$, let $d=x k$ for some $x\in (1/4,1/2)$. 
Note that $\frac{k-2d-1}{(k-1)^2}< \frac{k-2d}{k^2}$, then by definition, we have
\begin{align*}
g(k, x k) \le 1- \left(1-(1-x)(1-2x) \right) \left( 1-\frac{1}{k} \right)^{(1-x)k} = 1- \left(3x-2x^2 \right) \left( 1-\frac{1}{k} \right)^{(1-x)k}.
\end{align*}
Let $h(k, x): = 1- \left(3x-2x^2 \right) \left( 1-\frac{1}{k} \right)^{(1-x)k}$ and note that for $x\in (1/4, 1/2)$ and $k\ge 2$, $h(k,x)$ is decreasing on $x$ and $k$, respectively.
So we are done by noticing that $h(20, 0.42)<1/2$.
%Since $g'(20, 0.42)<1/2$, we have that for all $k\ge 20$, $x\in (0.42, 1/2)$, $g(k, x k)\le g'(k, x)<1/2$ and  we are done.
\end{proof}

\subsection{Perfect fractional matchings in hypergraphs}

As shown in \cite{AFHRRS, KOTo, TrZh15}, to get upper bounds on $m_d(k,n)$, it suffices to study so-called perfect fractional matchings.
A \emph{fractional matching} in a $k$-graph $H=(V, E)$ is a function $w:E\rightarrow [0,1]$, such that for each $v\in V$ we have $\sum_{e\in E:v\in e}w(v)\le 1$.
The \emph{size} of $w$ is $\sum_{e\in E} w(e)$ and we say $w$ is a \emph{perfect fractional matching} if it has size $|V|/k$.
For $s\in \mathbb{R}$, let $f_d^s(k,n)$ denote the smallest integer $m$ such that every $n$-vertex $k$-graph $H$ with $\delta_d(H)\ge m$ has a fractional matching of size $s$.
Note that the $k$-graph $G(n/k-1)$ shows that $f_{d}^{n/k}(k,n) \ge \left(1- \left(1- \frac{1}{k} \right)^{k-d} +o(1) \right)\binom{n-d}{k-d}$.
The equality was shown for $d=k-1$ in \cite{RRS09} (in fact, it is shown that $f_{k-1}^{n/k}(k,n) = \lceil n/k \rceil$), and for $k/2\le d< k-1$ in \cite{KOTo}.

As the key component of the proof of Theorem~\ref{thm:int_mat}, we show the following upper bound on $f_d^{n/k}(k,n)$ for $1\le d<k/2$.
Let $c_{k,d}^*:=\lim \sup_{n\rightarrow \infty} f_d^{n/k}(k,n)/\binom{n-d}{k-d}$.
%Note that the limit superior exists because $f_d^{n/k}(k,n)\le f_d^{(n-k)/k}(k,n-k) + o(n^{k-d})$ and $f_d^{n/k}(k,n)/\binom{n-d}{k-d}\in [0,1]$.

\begin{theorem}\label{thm:frac_mat}
Let $n, k\ge 3$, $1\le d< k/2$ be integers. Then
\begin{equation*}
f_d^{n/k}(k,n) \le \left( g(k,d) + o(1) \right) \binom{n-d}{k-d},
\end{equation*}
or, equivalently, $c_{k,d}^* \le g(k,d)$.
%\begin{equation}\label{eq:fd}
%c_{k,d}^* \le g(k,d).
%\end{equation}
\end{theorem}

Now Theorem~\ref{thm:int_mat} immediately follows from Theorem~\ref{thm:frac_mat} and the following theorem of Treglown and Zhao \cite[Theorem 2]{TrZh15}. %showed the following theorem on $m_d(k,n)$.
\begin{theorem}\cite{TrZh15}\label{thm:TrZh}
Fix integers $k,d$ with $d\le k-1$ and let $n\in k\mathbb{N}$. Then
\[
m_d(k,n) = \max \left\{ \delta(n,k,d)+1, (c^*_{k,d}+o(1))\binom{n-d}{k-d} \right\}.
\]
\end{theorem}

%Note that Theorem~\ref{thm:int_mat} immediately follows from Theorems~\ref{thm:frac_mat} and~\ref{thm:TrZh}.

%\subsection{Large matchings in $k$-uniform hypergraphs}
\subsection{The Erd\H{o}s Matching Conjecture}
The following classical conjecture is due to Erd\H{o}s \cite{erdos65} in 1965.
Here we prefer the notation from Extremal Set Theory, where a $k$-uniform family $\F\subseteq \binom{[n]}{k}$ is a collection of $k$-subsets of $[n]$ (so it is a $k$-graph).
Given a family $\F$, $\nu(\F)$ is the size of the maximum matching in $\F$.

\begin{conjecture}\cite{erdos65}\label{conj:erdos}
If $\F\subseteq \binom{[n]}{k}$ and $\nu(\F) = s$ such that $n\ge k(s + 1) - 1$ then
\[
|\F|\le \max \left\{ \binom{k(s + 1) - 1}{k}, \binom nk - \binom{n-s}k \right\}
\]
holds.
\end{conjecture}

The two quantities in the above conjecture come from the following two simple constructions.
\[
\A(k,s):= \binom{[k(s+1)-1]}{k}, \, \,  \A(n,1,s) := \left \{ A\in \binom{[n]}{k} : A\cap [s]\neq\emptyset \right\}.
\]
Note that $ \A(n,1,s)$ is isomorphic to $G(s)$.

The case $s=1$ is the classical Erd\H{o}s-Ko-Rado Theorem \cite{EKR}.
For $k=1$ the conjecture is trivial and for $k=2$ it was proved by Erd\H{o}s and Gallai \cite{ErGa59}.
For general $k\ge 3$, Erd\H{o}s \cite{erdos65} proved the conjecture for $n > n_0(k,s)$.
Bollob\'as, Daykin and Erd\H{o}s \cite{BDE76} improved $n_0(k,s)$ to $2s k^3$, which was subsequently lowered to $3s k^2$ by Huang, Loh and Sudakov \cite{HLS}.
The best known bound on $n_0$ is $(2k - 1)s + k$ by Frankl \cite{Frankl13}.
Recently, Conjecture~\ref{conj:erdos} is verified for $k=3$ by {\L}uczak and Mieczkowska \cite{LuMi14} for large $n$ and by Frankl \cite{Frankl12} for all $n$. 

Here we show a result from a different point of view. Instead of looking for exact solutions for smaller values of $n$, we give an upper bound on the size of the family for the unsolved cases.
Note that Frankl \cite{frankl87_survey} showed that $|\F|\le s\binom{n-1}{k-1}$ for all $n,k,s$.

\begin{theorem}\label{thm:main}
Suppose $n, k, s$ are non-negative integers and $\a\in (1,2-1/k]$ is a real number such that $n\ge \a k(s+1) + k-1$.
Let $\F\subseteq \binom{[n]}{k}$ and $\nu(\F) = s$ then
\[
|\F|\le \binom nk - \binom{n-s}k + \frac{(2-\a)k-1}{\a k-1}s\binom{n-s-1}{k-1}.
\]
\end{theorem}

Note that Theorem~\ref{thm:main} can be translated into the language of $m_0^s(k,n)$.
In fact, for any upper bound $h(n,k,s)$ on $|\F|$ where $\nu(\F)=s\le n/k-1$, we immediately have
\[
m_0^{s+1}(k,n) \le h(n,k,s)+1.
\]

%The rest of the note is organized as follows. We show Theorem~\ref{thm:int_mat} and Corollary~\ref{cor:mdk} in Section 2.
%We prove Theorems~\ref{thm:main} and \ref{thm:frac_mat} in Sections 3 and 4 respectively.

%\section{Proofs of Theorem~\ref{thm:int_mat} and Corollary~\ref{cor:mdk}}

%Moreover, our result also yields new exact values of $m_d(k,n)$.
%Here we show how Theorem~\ref{thm:frac_mat} provides new results on $m_d(k,n)$.

%Let $g(x) = \lim_{k\rightarrow \infty} g(k,x)$ and clearly $g(x) = 1- \left(3x-2x^2 \right) e^{(x-1)}$. Solving $g(x) = 1/2$ gives us that $x\approx 0.415$, which is the limitation of Theorem~\ref{thm:frac_mat}. 

\section{Proof of Theorem~\ref{thm:main}}

Our proof of Theorem~\ref{thm:main} is adapted from the proof of \cite[Theorem 1.1]{Frankl13}.
Let us first recall two results from \cite{Frankl13}.
For a family $\F\subset \binom{[n]}k$, its shadow is defined as
\[
\partial \F : = \left\{ G\in \binom{[n]}{k-1}: \exists F\in \F,\, G\subset F \right\}.
\]

\begin{theorem}\cite[Theorem 1.2]{Frankl13}\label{thm:shadow}
If $\F\subseteq \binom{[n]}{k}$ and $\nu(\F) = s$, then
\[
s|\partial{\F}|\ge |\F|.
\]
\end{theorem}

The families $\F_1, \F_2,\dots, \F_{s+1}$ are called \emph{nested} if $\F_{s+1}\subseteq \F_s \subseteq \cdots \subseteq \F_1$ holds.
The families $\F_1, \F_2,\dots, \F_{s+1}$ are called \emph{cross-dependent} if there is no choice of $F_i\in \F_i$ such that $F_1,\dots, F_{s+1}$ are pairwise disjoint.
Here we use a theorem in \cite{Frankl13} in a slightly different form, which follows from the original proof.

\begin{theorem}\cite[Theorem 3.1]{Frankl13}\label{thm:ineq}
Let $\beta\in (0,1)$ and let $\F_1,\F_2,\dots, \F_{s+1}\subseteq \binom{Y}{\ell}$, be nested, cross-dependent families, $|Y|\ge t\ell$. Suppose further that $t\ge \beta(2s+1)$, then
\[
|\F_1|+|\F_2|+\cdots+ |\F_{s}| + (s+1)|\F_{s+1}|\le \frac{s}\beta\binom{|Y|}{\ell}.
\]
\end{theorem}

%In a family $\F$ on $[n]$, we can define the following partial order. For two sets $F_1=\{a_1,\dots, a_k\}$ and $F_2=\{b_1,\dots, b_k\}$, define $F_1< F_2$ if $a_i< b_i$ for all $i\in [k]$ \footnote{One may define `$\le$', but a strict relation is enough here.}.
%We observe the following fact.
%
%\begin{fact}\label{fact:po}
%Let $\F$ be a stable family on $[n]$, then
%\begin{enumerate}
%\item if $F_2\in \F$ and $F_1< F_2$, then $F_1\in \F$.
%\item if $F_1\notin \F$ and $F_1< F_2$, then $F_2\notin \F$.
%\end{enumerate}
%\end{fact}

It is well known that in proving Theorem~\ref{thm:main} one can assume that $\F$ is \emph{stable}.
That is, for all $1\le i<j\le n$ and $F\in \F$, the conditions $i\notin F$, $j\in F$ imply that $F\cup \{i\}\setminus \{j\}$ is in $\F$ as well.

\begin{proof}[Proof of Theorem~\ref{thm:main}]
%Given $k\ge 2$, let $a\in (1,2-1/k)$ and let $n$ be a sufficiently large integer.
Let $\F\subset \binom{[n]}k$ be a stable family with $\nu(\F)=s$, $n\ge \a k (s+1) + k-1$.
%By Theorem~\ref{thm:frankl}, we may assume that $(2k-3)s + k\le n \le (2k-1)s + k$.
Note that $\a>1$ and thus $n > k (s+1)$.
We need to show that
\begin{equation*}%\label{eq:goal}
|\F|\le |\A(n,1,s)| + \frac{(2-\a)k-1}{\a k-1}s\binom{n-s-1}{k-1}.
\end{equation*}
Let us write $\A$ instead of $\A(n,1,s)$ throughout the proof.
In order to compare $\F$ and $\A$, we partition both families according to the intersection of their edges with $[s+1]$: For a subset $Q\subset [s+1]$ define
\begin{align*}
&\F(Q):= \{F\in \F: F\cap[s+1]=Q\},
&\A(Q):= \{A\in \A: A\cap[s+1]=Q\}.
\end{align*}
Let $m:=n-s-1$ and note that for $|Q|\ge 2$, we have $|\A(Q)|=\binom{m}{k-|Q|}$, which implies $|\F(Q)|\le |\A(Q)|$.
For $1\le i\le s$, $|\A(\{i\})|=\binom{m}{k-1}$ and $\A(\{s+1\})=\A(\emptyset)=\emptyset$.
Thus it suffices to show
\begin{equation}\label{eq:goal1}
|\F(\emptyset)| + \sum_{i=1}^{s+1} |\F(\{i\})| \le s\binom{m}{k-1} + \frac{(2-\a)k-1}{\a k-1}s\binom{m}{k-1} = \frac{2k-2}{\a k-1}s\binom{m}{k-1}.
\end{equation}

Note that $\nu(\F(\emptyset))\le s$ and $|\partial (\F(\emptyset))| \le |\F(\{s+1\})|$, where the latter is because every $H\in \partial (\F(\emptyset))$ satisfies that $H\cup\{s+1\}\in \F(\{s+1\})$.
Then by Theorem~\ref{thm:shadow}, we have
\[
|\F(\emptyset)|\le \nu(\F(\emptyset)) |\partial (\F(\emptyset))| \le s |\F(\{s+1\})|.
\]
%where we used $|\partial\F(\emptyset)| \le |\F(\{s+1\})|$, which is because every $H\in \partial\F(\emptyset)$ satisfies that $H\cup\{s+1\}\in \F(\{s+1\})$.
Plugging this into \eqref{eq:goal1}, we see that it suffices to show
\begin{equation}\label{eq:1}
|\F(\{1\})|+\cdots+|\F(\{s\})|+\left( s +1 \right)|\F(\{s+1\})|\le \frac{2k-2}{\a k-1} s \binom{m}{k-1}.
\end{equation}
To apply Theorem~\ref{thm:ineq} set $\F_i:=\{F\setminus \{i\}:F\in \F(\{i\})\}$.
Since $\F$ is stable, $\F_1,\dots, \F_{s+1}$ are nested.
Also, since $\nu(\F)=s$, $\F_1,\dots, \F_{s+1}$ are cross-dependent.
Setting $\ell:=k-1$, $Y:=[s+2, n]$ and thus
\[
|Y|=m\ge (\a k-1)(s+1)+k-1 = \left( \frac{\a k-1}{k-1} \frac{s+1}{2s+1} (2s+1)+1\right)(k-1) > \left \lceil \frac{\a k-1}{2k-2} (2s+1) \right\rceil (k-1).
\]
So all conditions of Theorem~\ref{thm:ineq} are satisfied for $t=\lceil\frac{\a k-1}{2k-2} (2s+1)\rceil$ and $\beta =\frac{\a k-1}{2k-2}$, and thus \eqref{eq:1} follows from Theorem~\ref{thm:ineq}.
\end{proof}

\section{Proof of Theorem~\ref{thm:frac_mat}}

Here we use the approach in \cite{KOTo} as well as two propositions. In fact, we only replace their \cite[Theorem 1.8]{KOTo} by our Theorem~\ref{thm:main}.

\begin{proposition}\cite[Proposition 4.1]{KOTo}\label{prop:41}
Let $k, d, n$ be integers with $n\ge k\ge 3$, and $1\le d\le k-2$.
Let $a\in [0, 1/k]$.
Suppose $H$ is a $k$-uniform hypergraph on $n$ vertices such that $\delta_d(H)\ge f_0^{a n}(k-d, n-d)$, then $H$ has a fractional matching of size $a n$.
\end{proposition}

\begin{proposition}\cite[Proposition 2.3]{KOTo}\label{prop:23}
Suppose that $k\ge 2$ and $a\in (0, 1/k)$, $c\in (0,1)$ are fixed.
Then for every $\e>0$ there exists $n_0=n_0(k,\e)$ such that if $n\ge n_0$ and $f_0^{a n}(k,n)\le c\binom n k$ then $f_0^{a n+1}(k,n)\le (c+\e)\binom n k$.
\end{proposition}

\begin{proof}[Proof of Theorem~\ref{thm:frac_mat}]
Let $k':=k-d$ and $n':=n-d$. Let $\a:=k/k'$ and $s+1:=\lfloor(n'-k'+1)/k\rfloor$, then applying Theorem~\ref{thm:main} with $n', k', s, \a$ implies that
\begin{align*}
m_0^{s+1}(k', n') &\le \binom {n'}{k'} - \binom{n'-s}{k'} + \frac{k-2 d-1}{k-1} s \binom{n'-s-1}{k'-1}+1  \\
&= \binom {n'}{k'} - \left(1- \frac{k-2 d-1}{k-1} s \cdot \frac{k-d}{n'-s} \right) \binom{n'-s}{k'}+1 = \left( g(k,d) +o(1)\right) \binom{n-d}{k-d}.
\end{align*}
Here the last equality is due to that $n'-s=(k-1+o(1)) s=(1-1/k+o(1))n'$, which follows from the definition of $s$.
Since $n/k \le s+3$, by Proposition~\ref{prop:23} and the trivial fact that $f_0^{s+1}(k', n')\le m_0^{s+1}(k', n')$, we get
\begin{align*}
f_0^{n/k}(k-d, n-d) =f_0^{n/k}(k', n')\le m_0^{s+1}(k', n')+o(1)\binom{n'}{k'} \le \left( g(k,d) +o(1)\right) \binom{n-d}{k-d}.
\end{align*}
So Theorem~\ref{thm:frac_mat} follows now from Proposition~\ref{prop:41}.
\end{proof}

\section{Acknowledgement}
We thank two anonymous referees for the comments and Andrew Treglown and Yi Zhao for helpful discussions and comments.

\bibliographystyle{plain}
\bibliography{Apr2015}

\end{document}